\documentclass[12pt]{article}

\usepackage{graphics,amssymb,amsmath,epsfig}

\newtheorem{theorem} {Theorem}[section]
\newtheorem{Definition} {Definition}[section]

\newtheorem{Corollary} {Corollary}[section]

\numberwithin{equation}{section}
\newenvironment{proof}{ {\setlength {\parindent} {0 pt} \bf Proof\ } }{ \hfill {$\square$} \par\vskip 7 pt}
\date{}
\usepackage{amsbsy}
\usepackage{amssymb,color}

\title{Certain Family of Some Beta Distributions Arising from Distribution of Randomly Weighted Average}
\author{Rasool Roozegar\thanks{Corresponding: rroozegar@yazd.ac.ir }\\
\noindent \footnotesize{Department of Statistics, Yazd University, P.O. Box 89195-741, Yazd, Iran} }
\date{}

\begin{document}
\maketitle \thispagestyle{empty}

\begin{abstract}
We give the exact distribution of the average of $n$ independent beta random variables weighted by the selected cuts of $(0,1)$ by the order statistics of a random sample of size $n-1$ from the uniform distribution $U(0,1)$, for each $n$. A new integral transformation that is similar to generalized Stieltjes transform is given with various properties. The result of Soltani and Roozegar [On distribution of randomly ordered uniform incremental weighted averages: Divided difference approach. Statist Probab Lett. 2012;82(5):1012–1020] with this new transform and also integral representation of the Gauss-hypergeometric function in some parts are employed to achieve the exact distribution. Several examples of the new family are investigated.
\end{abstract}

\section{Introduction}

Let $U_{(1)}< U_{(2)}< \ldots < U_{(n-1)}$ be order statistics based on a random sample of size $n-1$ from the uniform distribution $U(0,1)$, $U_{(0)}=0$ and $U_{(n)}=1$. Randomly weighted average (RWA) of independent and continuous random variables $X_{1},\ldots,X_{n}$ with respective distribution functions $F_{1},\ldots,F_{n}$, is defined by
\begin{equation}
S_n=R_{1}X_{1}+R_{2}X_{2}+\cdots +R_{n}X_{n},\ \ \ n\geq 2,
\end{equation}
where the proportions $R_{i}=U_{(i)}-U_{(i-1)},$ $i=1,\ldots ,n-1$ and $R_{n}=1-\sum_{i=1}^{n-1}R_{i}$ are random weights. Suppose that among the above whole order statistics, we select $k-1$ order statistics $U_{(n_1)}< U_{(n_2)}< \ldots < U_{(n_{k-1})}$, where $k=k(n)$, $2‎\leq k ‎\leq n‎‎$ and $n_0=0< n_1< n_2< ‎\cdots‎< n_{k-1}< n_k=n ‎$. A  general form to $S_n$ will be the RWA of $k$ independent and continuous random variables $X_1, ‎\ldots, X_k‎$, denoted by $S_{n: n_1,‎\ldots, n_{k-1}‎}$, which is given by
\begin{equation}
S_{n: n_1,‎\ldots, n_{k-1}‎}=\sum_{j=1}^{k} V_{j} X_{j},
\end{equation}
where the random weights $V_j$ are defined by
$$V_{j}=U_{(n_j)}-U_{(n_{j-1})},\ \ \ \ j=1,2,‎\ldots, k.$$
In the above we have put the conventions $U_{(n_0)}=0$, $U_{(n_k)}=1$ and $r_j=n_j - n_{j-1},\ j=1,2,‎\ldots,k‎$. Then $U_{(n_j)}=\sum_{i=1}^{j} V_i$ and $V_j=\sum_{i=n_{j-1}+1}^{n_j} R_i$.
The RWA on the form (1.2) was defined and studied by \cite{Soltani}. In fact the random vector $‎\textbf{V}‎=(V_1, V_2, ‎\ldots, V_k‎)$ has the Dirichlet distribution, $Dir (r_1, r_2, ‎\ldots, r_k‎)$, with the probability density function (pdf)
\begin{equation*}
f_{‎\textbf{V}‎}(v_1,‎\ldots, v_k‎)=‎\dfrac{‎\Gamma‎ (n)}{\prod‎_{j=1}^{k} \Gamma‎ (r_j)‎}‎ \prod_{j=1}^{k-1} v_j^{r_j} (1-\sum_{j=1}^{k-1} v_j)^{r_k},
\end{equation*}
at any point in the canonical simplex $\{(v_1,‎\ldots, v_k‎) | v_i‎\geq 0, \ i=1,2,‎\ldots, k, \ \sum_{j=1}^{k} v_j=1‎‎ \}$ in the $(k-1)$-dimensional real space $\mathbb{R}^{k-1}$ and zero outside.

To state assertions, we introduce beta distributions, $\tilde{b‎}eta(p,q‎‎)$‎, over $[a,b]$‎ by the pdf 
\begin{equation*}
f_{‎p,q}(x‎)=‎\left\{
\begin{array}{cc}
‎\dfrac{1}{B(p,q) (b-a)^{p+q-1}}‎ (x-a)^{p-1} (b-x)^{q-1} &
\mathrm{if}~ a<x<b, \\
0 &
\mathrm{other wise,}%
\end{array}%
\right.
\end{equation*}
for $ p, q>0$ and $B(p,q)=‎\dfrac{‎\Gamma‎ (p) ‎\Gamma (q)‎}{‎\Gamma‎ (p+q)}‎$ is beta function.\\
Let us denote the usual beta distribution over $[0,1]$ by $beta(p,q)$ with above pdf when $a=0$ and $b=1$. The power semicircle distribution with parameters $\theta$ and $\sigma$ on $(-\sigma, \sigma)$, denoted by $PS(\theta, \sigma)$, is a special case of beta distribution whenever $p=q=\theta+3/2‎$, $b=\sigma$ and $a=-\sigma$. We recall that $PS(-1, \sigma)$ is arcsine distribution on $(-\sigma, \sigma)$, for $\theta=-1/2$ the power semicircle density is the uniform on $(-\sigma, \sigma)$ and for $\theta=0$, it is the semicircle (Wigner) density on $(-\sigma, \sigma)$.

The ordinary Stieltjes transform (ST) and generalized Stieltjes transform (GST) of a distribution $H$ are respectively defined by 
\begin{equation*}
\mathcal{S}[H](z)=\int_{\mathbb{R}}\frac{1}{z-x}dH(x),\text{ \  \ }z\in \mathbb{C}\cap (\mbox{supp} \mathit{H})^{c},
\end{equation*}
and
\begin{equation*}
\mathcal{S}[H; \rho](z)=\int_{\mathbb{R}}\frac{1}{(z-x)^{\rho }}dH(x),\ \ z\in \mathbb{C}\cap (\mbox{supp} \mathit{H})^{c},\ \rho >0,
\end{equation*}
where $\mathbb{C}$ is the set of complex numbers, supp$H$ stands for the support of $H$ and $\rho$ is a constant. For more on the ST and GST, see \cite{Debnath}.

Soltani and Roozegar \cite{Soltani} effectively apply certain results in divided differences and derive the useful relation between GST of $F_{S_{n: n_1,‎\ldots, n_{k-1}}}$ and GSTs of $F_{1},\ldots,F_{n}$; more is given there:
\begin{equation}
\mathcal{S}[F_{S_{n: n_{1},...,n_{k-1}}}; n](z)=\prod_{i=1}^{k}
\mathcal{S}[F_{i}; r_i](z),\text{ \ \ \ }z\in \mathbb{C}{\ }{\bigcap_{i=1}^{k}}(\mbox{supp}\text{ }\mathit{F_{i}})^{c}.
\end{equation}
In particular
\begin{equation*}
\mathcal{S}[F_{S_{n}}; n](z)=\prod_{i=1}^{n}\mathcal{S}[F_{i}](z),\text{ \ \ \ }z\in \mathbb{C}{\ }{\bigcap_{i=1}^{n}}(\mbox{supp}\text{ }\mathit{F_{i}})^{c}.
\end{equation*}
The Gauss-hypergeometric function $F_{D}^{(1)}$, is defined by the series
\begin{equation*}
F_{D}^{(1)}(c,a;b;z)=\sum_{n=0}^{\infty }\frac{(c)_{n}(a)_{n}}{(b)_{n}n!}z^{n},
\end{equation*}
where $(a)_{0}=1$ and $(a)_{n}=a(a+1)(a+2)\cdots (a+n-1),$ $n\geq 1,$ denotes the rising factorial. Gauss-hypergeometric function
$F_{D}^{(1)}$ has the Euler's integral representation of the form
\begin{equation}
F_{D}^{(1)}(c,a;b;z)=\frac{\Gamma (b)}{\Gamma (a)\Gamma (b-a)}\int_{0}^{1}\frac{t^{a-1}(1-t)^{b-a-1}}{(1-zt)^{c}}dt.
\end{equation}
For more details on Gauss-hypergeometric function and its properties, see \cite{Abramowitz} and \cite{Andrews}.

There are few examples of known distributions that are distributions of RWA on the form (1.1), $S_n$, and there is not any examples, to the best of our knowledge, that are distributions of RWA on the form (1.2), $S_{n: n_{1},...,n_{k-1}}$. Roozegar and Soltani \cite{Roozegar} used the investigations of \cite{Demni} and \cite{Kubo} on the connection between ST and GST of two distributions and introduce new classes of power semicircle laws that are RWA $S_n$ distributions. In this paper we introduce new classes of RWA $S_{n: n_{1},...,n_{k-1}}$ distributions.\\
Rest of the paper is organized as follows. In Section 2 we introduce a new transform similar to GST for later analysis. We rewrite the main result of \cite{Soltani} based on this new transform as a Theorem 2.2 in this section. Two new classes of RWA $S_{n: n_{1},...,n_{k-1}}$ distributions are investigated in Section 3. In Section 4, we present some examples of the new classes.

\section{New transform and RWA distribution}

First, we define a new univariate characteristic function called an additive Stieltjes transform (AST).
\begin{Definition}
If $X$ is a random variable with distribution $H$ on a subset $S$ of $A=[-a,a]$, $a>0$, its AST is defined as
\begin{eqnarray*}
\mathcal{AS}[H; d](z) &=& E_X [‎\dfrac{1}{(1-z X)^d}‎] \\
&=& \int_{S} ‎\frac{1}{(1-z x)^d} ‎dH(x), \ \ |z|<‎\dfrac{1}{a}‎,
\end{eqnarray*}
where $d$ is a positive real number.
\end{Definition}
The assumptions that $d$ is positive and $H$ has a support in $S$ are needed for the one to one correspondence between AST and $H$ in next theorem.\\
We have the following relationship between the AST and the GST of a distribution $H$,
\begin{equation}
\mathcal{AS}[H; d](z)=\frac{1}{z^{d}}\mathcal{S}\left[ H;d \right] (\frac{1}{z}).
\end{equation}
\begin{theorem}
For distributions $H_1$ and $H_2$ in a subset $S$ of $A=[-a,a]$ and any positive real number $d$, if we have
\begin{equation}
\mathcal{AS}[H_1; d](z)= \mathcal{AS}[H_2; d](z),
\end{equation}
for all $|z|<‎\dfrac{1}{a}$, then $H_1=H_2$.
\end{theorem}
\begin{proof}
Since $d$ is a real number and $|zx|<1$, we have
\begin{equation}
(1-zx)^{-d}=\sum_{m=0}^{\infty}‎\dfrac{‎\Gamma‎ (d+m)}{‎\Gamma‎ (d) m!}‎ (zx)^m.
\end{equation}
By Definition 2.1 and equations (2.2) and (2.3), we have
\begin{equation}
\sum_{m=0}^{\infty}‎\dfrac{‎\Gamma‎ (d+m)}{‎\Gamma‎ (d) m!}‎ z^m \int_{S} x^{m} dH_1(x)= \sum_{m=0}^{\infty}‎\dfrac{‎\Gamma‎ (d+m)}{‎\Gamma‎ (d) m!}‎ z^m \int_{S} x^{m} dH_2(x),
\end{equation}
for all $|z|<‎\dfrac{1}{a}$. Given an integer $k$, treating $z$ as variable and equating the corresponding coefficients of $z^m$ ($\dfrac{‎\Gamma‎ (d+m)}{‎\Gamma‎ (d)}‎‎\neq 0‎$, for all $m$, as $d$ is positive) for each $m$ in the two sums, we obtain
\begin{equation*}
\int_{S} x^{m} dH_1(x)=\int_{S} x^{m} dH_2(x).
\end{equation*}
Hence we see that
\begin{equation*}
\int_{S} P(x) \ dH_1(x)=\int_{S} P(x) \ dH_2(x),
\end{equation*}
where $P(x)$ is any polynomial function of $x$, and similarly for any continuous function. Thus $H_1=H_2$.
\end{proof}

The following theorem enables us to represent a relationship between the AST of the distribution of RWA $S_{n: n_{1},...,n_{k-1}}$ and to those of $X_1,‎\ldots, X_k‎$. This theorem is the main theorem of \cite{Soltani} based on the AST.
\begin{theorem}
Let $S_{n: n_{1},...,n_{k-1}}$ be the RWA given in (1.2). Assume random variables $X_{1},\ldots, X_{k}$ are independent and
continuous with distribution functions $ F_{1},‎\ldots, F_{k}$, respectively. Then
\begin{equation*}
\mathcal{AS}[F_{S_{n: n_{1},...,n_{k-1}}}; \sum_{i=1}^{k}r_{i}](z)=\prod_{i=1}^{k}\mathcal{AS}\left[ F_{i};r_{i}\right] (z),
\end{equation*}
for all $|z|<‎\dfrac{1}{a}$.
\end{theorem}
\begin{proof}
Applying Theorem 3.1 of \cite{Soltani} and then (2.1) to conclude the result.
\end{proof}

\section{Families of RWA on some beta distributions}

In this section, we provide some families of RWA $S_{n: n_{1},...,n_{k-1}}$ distributions.
\begin{theorem}
Let $X_i, \ i=1,2,‎\ldots,k$ be independent random variables with $\tilde{b‎}eta(r_i+1/2,r_i+1/2‎‎)$ distributions. Then the RWA $S_{n: n_{1},...,n_{k-1}}$ has  $\tilde{b‎}eta(\sum_{i=1}^{k} r_i +1/2,\sum_{i=1}^{k} r_i‎‎+1/2)$ distribution.
\end{theorem}
\begin{proof}
By definition of AST of beta distribution on $[a,b]$, we have
\begin{equation*}
\mathcal{AS}\left[F; r_{i}\right] (z)=\int_{a}^{b} ‎\dfrac{1}{(1-zx)^{r_i}}‎ ‎\frac{(x-a)^{r_i-1/2} (b-x)^{r_i-1/2}}{B(r_i+1/2, r_i+1/2) (b-a)^{2r_i}}‎ dx.
\end{equation*}
By the change of variable $t=‎\dfrac{2x-(b+a)}{b-a}‎$, we obtain that
\begin{equation*}
\mathcal{AS}\left[F; r_{i}\right] (z)= ‎\dfrac{1}{[1-(a+b)z/2]^{r_i}}‎\int_{-1}^{1} ‎\dfrac{1}{(1-‎\dfrac{z(b-a)}{2-(a+b)z}‎t)^{r_i}}‎ ‎\frac{(t+1)^{r_i-1/2} (1-t)^{r_i-1/2}}{2^{2r_i} B(r_i+1/2, r_i+1/2) }‎ dt.
\end{equation*}
By Table 3 in \cite{Kubo}, it follows that
\begin{eqnarray*}
\mathcal{AS}\left[F; r_{i}\right] (z) &=& ‎\dfrac{1}{[1-(a+b)z/2]^{r_i}}‎ ‎\left[ ‎\dfrac{2}{1+‎\sqrt{1-‎\dfrac{z^2 (b-a)^2}{(2-(a+b)z)^2}‎}‎}‎ ‎\right] ^{r_i} \\
&=& ‎\left[ ‎\dfrac{4}{2-(a+b)z+2 ‎\sqrt{1-(a+b)z+ab z^2}‎}‎ ‎\right] ^{r_i}‎‎.
\end{eqnarray*}
Using Theorem 2.2 we obtain
\begin{equation*}
\mathcal{AS}[F_{S_{n: n_{1},...,n_{k-1}}}; \sum_{i=1}^{k}r_{i}](z)= \left[ ‎\dfrac{4}{2-(a+b)z+2 ‎\sqrt{1-(a+b)z+ab z^2}‎}‎ ‎\right] ^{\sum_{i=1}^{k} r_i},
\end{equation*}
which is the AST of $\tilde{b‎}eta(\sum_{i=1}^{k} r_i +1/2,\sum_{i=1}^{k} r_i‎‎+1/2)$ distribution. Therefore ${S_{n: n_{1},...,n_{k-1}}}$ has $\tilde{b‎}eta(\sum_{i=1}^{k} r_i +1/2,\sum_{i=1}^{k} r_i‎‎+1/2)$ distribution.
\end{proof}

Another general family of RWA on beta distributions on $[a,b]$ is presented in the following theorem.
\begin{theorem}
Let $X_i, \ i=1,2,‎\ldots,k$ be independent random variables with $\tilde{b‎}eta(s_i, r_i-s_i‎‎)$ distributions. Then the RWA $S_{n: n_{1},...,n_{k-1}}$ has  $\tilde{b‎}eta(\sum_{i=1}^{k} s_i, \sum_{i=1}^{k} r_i‎‎- \sum_{i=1}^{k} s_i)$ distribution.
\end{theorem}
\begin{proof}
By using AST of $\tilde{b‎}eta(s_i, r_i-s_i‎‎)$ distribution, we have
\begin{equation*}
\mathcal{AS}\left[F; r_{i}\right] (z)= ‎\int_{a}^{b} ‎\dfrac{1}{(1-zx)^{r_i}}‎ ‎\frac{(x-a)^{s_i-1} (b-x)^{r_i- s_i-1}}{B(s_i, r_i-s_i) (b-a)^{r_i-1}}‎ dx.
\end{equation*}
Changing of variable $t=‎\dfrac{x-a}{b-a}‎$ implies
\begin{equation*}
\mathcal{AS}\left[F; r_{i}\right] (z)= ‎\dfrac{1}{(1-za)^{r_i}} ‎\int_{0}^{1} ‎\dfrac{1}{(1-‎\dfrac{z(b-a)}{1-za}‎t)^{r_i}}‎ ‎\frac{t^{s_i-1} (1-t)^{r_i- s_i -1}}{B(s_i, r_i-s_i)}‎ dt.
\end{equation*}
Using (1.4), it follows that
\begin{eqnarray*}
\mathcal{AS}\left[F; r_{i}\right] (z) &=& ‎\dfrac{1}{(1-za)^{r_i}}‎ F_{D}^{(1)} (r_i,s_i;r_i;‎\dfrac{z(b-a)}{1-za}‎) \\
&=& ‎\dfrac{1}{(1-za)^{r_i}} ‎\dfrac{1}{(1-‎\dfrac{z(b-a)}{1-za})^{s_i}}‎ \\
&=& ‎\dfrac{1}{(1-za)^{r_i-s_i}} ‎\dfrac{1}{(1-zb)^{s_i}}.
\end{eqnarray*}
From Theorem 2.2, we get
\begin{equation*}
\mathcal{AS}[F_{S_{n: n_{1},...,n_{k-1}}}; \sum_{i=1}^{k}r_{i}](z)= \dfrac{1}{(1-za)^{\sum_{i=1}^{k} (r_i-s_i)}} ‎\dfrac{1}{(1-zb)^{\sum_{i=1}^{k} s_i}}.
\end{equation*}
which is the AST of $\tilde{b‎}eta(\sum_{i=1}^{k} s_i ,\sum_{i=1}^{k} r_i‎‎-\sum_{i=1}^{k} s_i)$ distribution. Hence ${S_{n: n_{1},...,n_{k-1}}}$ has $\tilde{b‎}eta(\sum_{i=1}^{k} s_i ,\sum_{i=1}^{k} r_i‎‎-\sum_{i=1}^{k} s_i)$ distribution.
\end{proof}

The following corollary is an immediate consequence of Theorem 3.2.
\begin{Corollary}
For independent random variables $X_i, \ i=1,2,‎\ldots, k‎$ with $\tilde{b‎}eta(r_i/2, r_i/2)$ distributions, the RWA $S_{n: n_{1},...,n_{k-1}}$ has $\tilde{b‎}eta(\sum_{i=1}^{k} r_i/2, \sum_{i=1}^{k} r_i/2)$ distribution.
\end{Corollary}

The following theorem of \cite{Roozegar} comes as a corollary to Corollary 3.1.
\begin{Corollary}
For every integer $n‎‎\geq 2‎‎$, a power semicircle distribution with shape parameter $\theta=(n-3)/2$ and any positive range parameter is a RWA $S_n$ distribution.
\end{Corollary}
\begin{proof}
Consider $k=n$, $r_i=1$ for $i=1,2,‎\ldots, k‎$, $p=q=1/2$, $a=-\sigma$ and $b=\sigma$, then use Corollary 3.1 to obtain the result.
\end{proof}

\section{Examples}

Theorem 3.1 and Theorem 3.2 in Section 3 give general results on distribution of RWA $S_{n: n_{1},...,n_{k-1}}$ on some beta distributions on $[a,b]$. In this section we present some special examples of distributions that are randomly weighted average distributions given in (1.2). Other examples of RWA $S_n$ distributions in case of $r_1=r_2=‎\cdots=r_k=1‎$ and $k=n$ can be found in \cite{Roozegar}. For the first and second examples, we use the results of \cite{Kubo}.\\
\newline
\noindent\textbf{Example 4.1.} Let $X_1$, $X_2$ and $R$ be independent random variables with $\tilde{b‎}eta(1/2, 1/2)$, $\tilde{b‎}eta(m-1/2, m-1/2)$ and $beta(1,m-1)$ distributions, respectively. Then randomly weighted average $S=R X_1+(1-R)X_2$ has $\tilde{b‎}eta(m-1/2, m-1/2)$ distribution. Its AST can be written as
\begin{equation*}
\mathcal{AS}\left[ F_{S};m \right] (z)=\frac{1}{‎\sqrt{1-z(a+b)+ ab z^2}‎} ‎\left[ ‎\dfrac{4}{2-z(a+b)+2 ‎\sqrt{1-z(a+b)+ ab z^2}}‎‎ \right] ^{m-1}.‎‎
\end{equation*}
\newline

\noindent\textbf{Example 4.2.} Let $X_1$, $X_2$ and $R$ be independent random variables with $\tilde{b‎}eta(3/2, 1/2)$, $\tilde{b‎}eta(m-1/2, m-1/2)$ and $beta(1,m-1)$ distributions, respectively. Then randomly weighted average $S=R X_1+(1-R)X_2$ has $\tilde{b‎}eta(m+1/2, m-1/2)$ distribution. Its AST is given by
\begin{equation*}
\mathcal{AS}\left[ F_{S};m \right] (z)=\frac{2}{‎\sqrt{1-z ab+ \sqrt{1-z(a+b)+ ab z^2}}}‎ ‎\left[ ‎\dfrac{4}{2-z(a+b)+2 ‎\sqrt{1-z(a+b)+ ab z^2}}‎‎ \right] ^{m-1}.‎‎
\end{equation*}
\newline

\noindent\textbf{Example 4.3.} Let $X_1, X_2, ‎\ldots, ‎X_k$ be independent and identically distributed (i.i.d) random variables with common $U(a,b)$ distribution and random vector $(V_1,‎\ldots, V_k)‎$ has $Dir (2,2,‎\ldots, 2‎)$ distribution. Then by Corollary 3.1, the RWA $S_{n: n_{1},...,n_{k-1}}$ has $\tilde{b‎}eta(k, k)$ distribution.
\newline

\noindent\textbf{Example 4.4.} Let $X_1, X_2, ‎\ldots, ‎X_k$ be i.i.d random variables with common semicircle (Wigner) distribution on $(-\sigma, \sigma)$ and random vector $(V_1,‎\ldots, V_k)‎$ has $Dir (3,3,‎\ldots, 3‎)$ distribution. Then by Corollary 3.1, the RWA $S_{n: n_{1},...,n_{k-1}}$ has $\tilde{b‎}eta(3k/2, 3k/2)$ distribution.
\newline

\end{document}